\documentclass[10]{article}
\usepackage{amsfonts,amssymb,amsmath,amsthm,cite}
\usepackage{graphicx}

\usepackage[applemac]{inputenc}
\usepackage{amsmath,amssymb,amsthm, hyperref, euscript}
\usepackage[matrix,arrow,curve]{xy}
\usepackage{graphicx}
\usepackage{tabularx}
\usepackage{float}
\usepackage{hyperref}
\usepackage{tikz}
\usepackage{slashed}
\usepackage{mathrsfs}

\usetikzlibrary{matrix}

\usepackage[T1]{fontenc}
\usepackage{amsfonts,cite}
\usepackage{graphicx}

\usepackage[applemac]{inputenc}

\usepackage[sc]{mathpazo}
\DeclareFontFamily{T1}{pzc}{}
\DeclareFontShape{T1}{pzc}{m}{it}{1.8 <-> pzcmi8t}{}
\DeclareMathAlphabet{\mathpzc}{T1}{pzc}{m}{it}

\theoremstyle{plain}
\newtheorem{prop}{Proposition}[section]

\newtheorem{lem}[prop]{Lemma}
\newtheorem{thm}[prop]{Theorem}

\theoremstyle{definition}
\newtheorem{defn}[prop]{Definition}
\newtheorem{exm}[prop]{Example}
\newtheorem{rem}[prop]{Remark}

\newcommand{\vertiii}[1]{{\left\vert\kern-0.25ex\left\vert\kern-0.25ex\left\vert #1
    \right\vert\kern-0.25ex\right\vert\kern-0.25ex\right\vert}}

\usepackage[sc]{mathpazo}
\newbox\ncintdbox \newbox\ncinttbox 
\setbox0=\hbox{$-$} \setbox2=\hbox{$\displaystyle\int$}
\setbox\ncintdbox=\hbox{\rlap{\hbox
    to \wd2{\hskip-.125em \box2\relax\hfil}}\box0\kern.1em}
\setbox0=\hbox{$\vcenter{\hrule width 4pt}$}
\setbox2=\hbox{$\textstyle\int$} \setbox\ncinttbox=\hbox{\rlap{\hbox
    to \wd2{\hskip-.175em \box2\relax\hfil}}\box0\kern.1em}

\newcommand{\Id}{\mathrm{Id}}                

\newcommand{\A}{\mathcal{A}}                 

\newcommand{\C}{\mathbb{C}}                  
\newcommand{\hookto}{\hookrightarrow}        

\newcommand{\N}{\mathbb{N}}                  
\newcommand{\R}{\mathbb{R}}                  

\newcommand{\Z}{\mathbb{Z}}                  

\newcommand{\al}{\alpha}          
\newcommand{\bt}{\beta}           

\def\<#1|#2>{\langle#1\stroke#2\rangle} 
\def\?#1|#2?{\{#1\stroke#2\}}        

\def\<#1,#2>{\langle#1,#2\rangle}            
\def\ee_#1{e_{{\scriptscriptstyle#1}}}       
\def\wick:#1:{\mathopen:#1\mathclose:}       

\newbox\ncintdbox \newbox\ncinttbox 
\setbox0=\hbox{$-$}
\setbox2=\hbox{$\displaystyle\int$}
\setbox\ncintdbox=\hbox{\rlap{\hbox
           to \wd2{\box2\relax\hfil}}\box0\kern.1em}
\setbox0=\hbox{$\vcenter{\hrule width 4pt}$}
\setbox2=\hbox{$\textstyle\int$}
\setbox\ncinttbox=\hbox{\rlap{\hbox
           to \wd2{\hskip-.05em\box2\relax\hfil}}\box0\kern.1em}

\newcommand{\stroke}{\mathbin|}   

\newcommand{\SU}{SU}       

\title{Coverings of Quantum Groups}
\begin{document}
\maketitle  \setlength{\parindent}{0pt}
\begin{center}
\author{
{\textbf{Petr R. Ivankov*}\\
e-mail: * monster.ivankov@gmail.com \\
}
}
\end{center}



\paragraph{}
It is known that any covering space of a topological group has the natural structure of a topological group. This article discusses a noncommutative generalization of this fact. A noncommutative generalization of the topological group is a quantum group. Also there is a noncommutative generalization of a covering. The combination of these algebraic constructions yields a motive to research the generalization of coverings of topological groups. In contrary to a topological group a covering space of a quantum group does not have the natural structure of the quantum group. However a covering space of a quantum group satisfies to a condition which is weaker than the condition of a covering space of a topological group.

\section{Motivation. Preliminaries}

\paragraph{}  
In this article we discuss a noncommutative analog of the following proposition.
\begin{prop}\label{cov_top_prop}\cite{mimuta_toda_lie} If $G$ is a topological group and $\pi:\widetilde{G} \to G$ is a covering, then for a covering space $\widetilde{G}$ one can introduce uniquely the structure of a topological group on $\widetilde{G}$ such that $\pi$ is a homomorphism and an arbitrary point $\widetilde{e}$ of the fibre over the unit $e$ of $G$ is the unit.
\end{prop}
For this purpose we need noncommutative generalizations of following objects:
\begin{itemize}
	\item Topological spaces,
	\item Coverings,
	\item Topological groups. 
\end{itemize}
\subsection{Generalization of topological objects}
\subsubsection{Noncommutative topological spaces}
  \paragraph*{} Gelfand-Na\u{\i}mark theorem \cite{arveson:c_alg_invt} states the correspondence between  locally compact Hausdorff topological spaces and commutative $C^*$-algebras.

\begin{thm}\label{gelfand-naimark}\cite{arveson:c_alg_invt} (Gelfand-Na\u{\i}mark). 
Let $A$ be a commutative $C^*$-algebra and let $\mathcal{X}$ be the spectrum of A. There is the natural $*$-isomorphism $\gamma:A \to C_0(\mathcal{X})$.
\end{thm}

\paragraph*{}So any (noncommutative) $C^*$-algebra may be regarded as a generalized (noncommutative)  locally compact Hausdorff topological space. 
\subsubsection{Generalization of coverings}
\paragraph*{} Following theorem gives a pure algebraic description of finite-fold coverings of compact spaces.
\begin{thm}\label{pavlov_troisky_thm}\cite{pavlov_troisky:cov}
	Suppose $\mathcal X$ and $\mathcal Y$ are compact Hausdorff connected spaces and $p :\mathcal  Y \to \mathcal X$
	is a continuous surjection. If $C(\mathcal Y )$ is a projective finitely generated Hilbert module over
	$C(\mathcal X)$ with respect to the action
	\begin{equation*}
	(f\xi)(y) = f(y)\xi(p(y)), ~ f \in  C(\mathcal Y ), ~ \xi \in  C(\mathcal X),
	\end{equation*}
	then $p$ is a finite-fold  covering.
\end{thm}
\begin{defn}
	If $A$ is a $C^*$- algebra then an action of a group $H$ is said to be {\it involutive } if $ga^* = \left(ga\right)^*$ for any $a \in A$ and $g\in H$. Action is said to be \textit{non-degenerated} if for any nontrivial $g \in H$ there is $a \in A$ such that $ga\neq a$. 
\end{defn}
Following definition is motivated by the Theorem \ref{pavlov_troisky_thm}.
\begin{defn}\cite{ivankov:qnc}\label{fin_def_uni}
	Let $A \hookto \widetilde{A}$ be an injective *-homomorphism of unital $C^*$-algebras. Suppose that there is a non-degenerated involutive action $H \times \widetilde{A} \to \widetilde{A}$ of finite group, such that $A = \widetilde{A}^H\stackrel{\text{def}}{=}\left\{a\in \widetilde{A}~|~ a = g a;~ \forall g \in H\right\}$. There is an $A$-valued product on $\widetilde{A}$ given by
	\begin{equation*}\label{finite_hilb_mod_prod_eqn}
	\left\langle a, b \right\rangle_{\widetilde{A}}=\sum_{g \in H} g\left( a^*, b\right) 
	\end{equation*}
	and $\widetilde{A}$ is an $A$-Hilbert module. We say that $\left(A, \widetilde{A}, H \right)$ is an \textit{unital noncommutative finite-fold  covering} if $\widetilde{A}$ is a finitely generated projective $A$-Hilbert module.
\end{defn}
\subsubsection{Generalization of topological groups}
\paragraph*{} A compact quantum group can be regarded as a noncommutative analog of a compact topological group.
\begin{defn}\cite{neshveyev_tuset_qg}
(Woronowicz) A \textit{compact quantum group} is a pair $\left( A, \Delta\right) $, where $A$
is an unital $C^*$ -algebra and $\Delta: A \to A\otimes A $  is an unital *-homomorphism, called \textit{comultiplication}, such that
\begin{enumerate}
	\item [(a)] $\left( \Delta \otimes \Id_A\right) \Delta = \left( \Id_A \otimes \Delta \right) \Delta$ as homomorphisms $A \to A \otimes A  \otimes A$, (\textit{coassociativity});
\item[(b)] The spaces $\left(A \otimes 1 \right) \Delta A = \mathrm{span}\left\{\left(a \otimes 1\right)  \Delta \left( b \right)~|~a,b \in A \right\}$   and $\left(1 \otimes A\right) \Delta A$  are dense in $A\otimes A$  (\textit{cancellation property}).
\end{enumerate}
In this definition by the tensor product of $C^*$ -algebras we  mean the minimal tensor product.
\end{defn}
Following example shows that a compact topological group is a special case of a quantum group.

\begin{exm}\label{comm_exm}\cite{neshveyev_tuset_qg}
	Let $G$ be a compact group. Take $A$ to be the $C^*$-algebra $C\left(G\right)$ of continuous functions on $G$. Then $A \otimes A = C\left(G \times G\right)$, so we can define $\Delta$ by
	$$
	\Delta\left(f\right)\left(g,h\right) = f\left(gh\right) \text{ for all } g, h \in G
	$$
	Coassociativity of $\Delta$ follows from associativity of the product in $G$. To see that the cancellation property holds, note that $\left(A \otimes 1 \right) \Delta A$ is the unital $C^*$-subalgebra of $C\left(G\times G\right)$	spanned by all functions of the form $\left(g, h\right) \mapsto  f_1\left(g\right)f_2\left(gh\right)$. Since such functions separate points of $G\times G$, the $C^*$-algebra $\left(A \otimes 1 \right) \Delta A$ is dense in $C\left(G\times G\right)$ by the Stone-Weierstrass
	theorem.
	Any compact quantum group $\left(A, \Delta\right)$ with abelian $A$ is of this form. Indeed, by the Gelfand theorem, $A = C\left(G\right)$ for a compact space $G$. Then, since $A\otimes A = C\left(G\times G\right)$, the
	unital *-homomorphism $\Delta$  is defined by a continuous map $G\times G \to G$. Coassociativity
	means that
	$$
	f\left(\left(gh\right)k\right) = f\left(g\left(hk\right)\right) \text{ for all } f \in C\left(G\right),
	$$
	whence $\left(gh\right)k = g\left(hk\right)$, so $G$ is a compact semigroup. If $gh = gk$, then $f_1\left(g\right)f_2\left(gh\right) =
	f_1\left(g\right)f_2\left(gk\right)$ for all $f_1; f_2 \in C\left(G\right)$. By the cancellation property the functions of the
	form $\left(g',h'\right) \mapsto f_1\left(g'\right)f_2\left(g'h'\right)$ span a dense subspace of $C\left(G \times G\right)$. It follows that
	$f\left(g, h\right) = f\left(g, k \right)$ for all $f \in   C\left(G\times G\right)$, whence $h = k$. Similarly, if $hg = kg$, then $h = k$.
	Thus $G$ is a semigroup with cancellation.
	In \cite{neshveyev_tuset_qg} it is proven that that any compact semigroup with cancellation is a group.
\end{exm}

\subsection{Finite Galois coverings}\label{fin_gal_cov_sec}
\paragraph*{} Here I follow to \cite{auslander:galois}. Let $A \hookto \widetilde{A}$ be an injective homomorphism of unital algebras, such that
\begin{itemize}
	\item $\widetilde{A}$ is a projective finitely generated $A$-module,
	\item There is an action $G \times \widetilde{A} \to \widetilde{A}$ of a finite group $G$ such that $$A = \widetilde{A}^G=\left\{\widetilde{a}\in \widetilde{A}~|~g\widetilde{a}=\widetilde{a}; ~\forall g \in G\right\}.$$
\end{itemize}
Let us consider the category $\mathscr{M}^G_{\widetilde{A}}$ of $G-\widetilde{A}$ modules, i.e.  any object $M \in \mathscr{M}^G_{\widetilde{A}}$ is a $\widetilde{A}$-module with equivariant action of $G$, i.e. for any $m \in M$ a following condition holds
$$
g\left(\widetilde{a}m \right)=  \left(g\widetilde{a} \right) \left(gm \right) \text{ for any } \widetilde{a} \in \widetilde{A}, ~ g \in G.
$$
Any morphism $\varphi: M \to N$ in the category $\mathscr{M}^G_{\widetilde{A}}$ is $G$- equivariant, i.e.
$$
\varphi\left( g m\right)= g \varphi\left( m\right)   \text{ for any } m \in M, ~ g \in G.
$$
Let $\widetilde{A}\left[ G\right]$ be an algebra such that $\widetilde{A}\left[ G\right] \approx \widetilde{A}\times G$ as an Abelian group and a multiplication law is given by
$$
\left( a, g\right)\left( b, h\right) =\left(a\left(gb \right), gh  \right).
$$
The category $\mathscr{M}^G_{\widetilde{A}}$ is equivalent to the category $\mathscr{M}_{\widetilde{A}\left[ G\right]}$ of $\widetilde{A}\left[ G\right]$ modules. Otherwise in \cite{auslander:galois} it is proven the equivalence between a category $\mathscr{M}_{A}$ of $A$-modules and the category $\mathscr{M}_{\widetilde{A}\left[ G\right]}$. It turns out that the category $\mathscr{M}^G_{\widetilde{A}}$ is equivalent to the category $\mathscr{M}_{A}$.

\section{Main result}
\paragraph*{} 
From the Proposition \ref{cov_top_prop}, Theorem \ref{pavlov_troisky_thm} and Example \ref{comm_exm} it turns out the following lemma

\begin{lem}\label{comm_lem}
	Let $\left(A, \Delta \right)$ be a commutative compact quantum group, and let   $\left(A, \widetilde{A},  H \right)$ be a noncommutative finite-fold covering such that $\widetilde{A}$ is a commutative algebra. Following condition holds:
	\begin{enumerate}
		\item[(i)] There is the natural structure   $\left(\widetilde{A}, \widetilde{\Delta} \right)$ of the compact quantum group, such that
		$$
		\widetilde{\Delta}\left( a\right) = \Delta\left(a \right) \text{ for any } a \in A.
		$$
		\item[(ii)] Operation $\Delta$ is $H$-equivariant, i.e. from
		\begin{equation*}
		\widetilde{\Delta}\left(\widetilde{a} \right)= \sum_{\iota \in I}\widetilde{b}_\iota \otimes \widetilde{c}_\iota
		\end{equation*}
		it turns out that for any $g \in H$ following condition holds
		\begin{equation*}
		\begin{split}
		\widetilde{\Delta}\left(g\widetilde{a} \right) =\sum_{\iota \in I}g\widetilde{b}_\iota \otimes g\widetilde{c}_\iota.
		\end{split}
		\end{equation*}
	\end{enumerate}
\end{lem}
\begin{proof}
Indeed this lemma is an algebraic interpretation of the topological Proposition \ref{cov_top_prop}. 
\end{proof}
\paragraph*{}
The Lemma \ref{comm_lem} is not true in general, there is a counterexample described in the Section \ref{counter_sec}. However any quantum group satisfies to a following theorem.

\begin{thm}\label{main_thm}
Let $\left(A, \Delta \right)$ be a  quantum group. Let
$\left(A, \widetilde{A},  H \right)$ be a noncommutative finite-fold covering projection. There are natural $A$-bimodule morphisms
\begin{equation*}
\begin{split}
\Delta_L : \widetilde{A} \to \widetilde{A} \otimes A, \\
\Delta_R : \widetilde{A} \to A \otimes \widetilde{A}.
\end{split}
\end{equation*}
such that following conditions hold:
\begin{enumerate}

\item[(i)] Above morphisms are $H$-equivariant, i.e. for any $g \in H$ from
\begin{equation*}
\begin{split}
\Delta_L \left( \widetilde{a}\right) = \sum_{\iota \in I}\widetilde{b}_\iota \otimes c_\iota,\\
\Delta_R \left( \widetilde{a}\right) = \sum_{\iota \in I}c_\iota \otimes\widetilde{d}_\iota
\end{split}
\end{equation*}
it turns out that
\begin{equation*}
\begin{split}
\Delta_L \left(g \widetilde{a}\right) = \sum_{\iota \in I}g\widetilde{b}_\iota \otimes c_\iota,\\
\Delta_R \left(g \widetilde{a}\right) = \sum_{\iota \in I}c_\iota \otimes g\widetilde{d}_\iota.
\end{split}
\end{equation*}
	\item [(ii)] If $a \in A$ then
\begin{equation*}
\begin{split}
\Delta_L \left(a\right) = \Delta\left( a\right) ,\\
\Delta_R \left(a\right) = \Delta\left( a\right).
\end{split}
\end{equation*}
\end{enumerate}
\end{thm}
\begin{proof}
(i) If we apply to $\Delta: A \to A\otimes A$ a functor $ \widetilde{A}\otimes_A-$ then we have
$$
\Delta_L: \widetilde{A} \to \widetilde{A} \otimes A
$$
From the Section \ref{fin_gal_cov_sec} it follows that $\Delta_L$  is left $H$-equivariant. Similarly one can construct $\Delta_R$.
\newline
(ii) Follows from the definition of functors $ \widetilde{A}\otimes_A-$ and $ -\otimes_A \widetilde{A}$ and from that the *-homomorphism $A \hookto \widetilde{A}$ is injective.
\end{proof} 
\begin{rem}
	The statement of Theorem \ref{main_thm} is weaker than the statement of the Lemma \ref{comm_lem}. In fact the Theorem \ref{main_thm} describes a left and right action of the group $G$ on the quotient group $\widetilde{G}/H \approx G$.
\end{rem}

\section{Counterexample}\label{counter_sec}
\paragraph*{}
The counterexample of the Lemma \ref{comm_lem} is discussed here. 
\subsection{Noncommutative quantum $SU(2)$ group}

\paragraph{} Let $q$ be a real number such that $0<q<1$. 
A quantum group $C\left( \SU_q(2)\right) $ is an universal $C^*$-algebra algebra generated by two elements $\al$ and $\beta$ satisfying following relations:
\begin{equation}\label{su_q_2_rel_eqn}
\begin{split}
\al^*\al + \beta^*\beta = 1, ~~ \al\al^* + q^2\beta\beta^* =1,
\\
\al\bt - q \bt\al = 0, ~~\al\bt^*-q\bt^*\al = 0,
\\
\bt^*\bt = \bt\bt^*.
\end{split}
\end{equation}
\paragraph*{} The structure of the quantum group on $C\left( SU_q\left( 2\right)\right) $ is given by
\begin{equation}\label{su_2_qgr_eqn}
\begin{split}
\Delta(\alpha)=\alpha\otimes\alpha-q\beta^*\otimes\beta,\\
\Delta(\beta)=\beta\otimes\alpha+\alpha^*\otimes\beta.
\end{split}
\end{equation}

From  $C\left( SU_1\left(2 \right)\right) \approx C\left(SU\left(2 \right)  \right)$ it follows that  $C\left( \SU_q(2)\right) $ can be regarded as a noncommutative deformation of $SU(2)$. It is proven in \cite{woronowicz:su2} that the spectrum of $\bt\bt^*$ is the discrete set
$$
\left\{1, q^2, q^4, q^6, ..., 0\right\} \subset \C.
$$
If $n \in \N^0$ and  $f_n: \R \to \R$ is a continuous function such that 
$$
f_n(t)=\left\{
\begin{array}{c l}
0 & t \le q^{2n + 1} \\
0 & t \ge q^{2n - 1} \\
1 & t = q^{2n}
\end{array}\right..
$$
then $p^\al_n = f_n\left(\bt\bt^* \right) \in C\left( \SU_q(2)\right)$ is a projection.  
Let $Q, S \in B\left( \ell_2\left(\N^0 \right)\right) $ be given by 
\begin{equation*}
\begin{split}
Qe_k= q^ke_k, \\
Se_k = \left\{
\begin{array}{c l}
e_{k-1} & k > 0 \\
0 & k = 0
\end{array}\right.,
\end{split}
\end{equation*}
and let $R \in B\left( \ell_2\left(\Z \right)\right) $ be given by $e_k \mapsto e_{k+1}$.
There is  a faithful representation   $C\left(\SU_q\left( 2\right) \right) \to B\left(\ell_2\left(\N^0 \right) \otimes \ell_2\left(\Z \right) \right)  $ \cite{woronowicz:su2} given by

\begin{equation}\label{su_2_q_repr_eqn}
\begin{split}
\al \mapsto S\sqrt{1 - Q^2} \otimes 1_{B\left(\ell_2\left(\Z \right) \right) }, \\
\bt \mapsto Q \otimes R.
\end{split}
\end{equation}

If $R_\R \in B\left(  L^2\left(\R \right)\right) $ is given by
$$
R_\R\left(  \xi\right)  = e^{2\pi ix}\xi; \text{ where } e^{2 \pi ix} \in C_b\left(\R \right) 
$$
then similarly to \eqref{su_2_q_repr_eqn} one has a representation $C\left(\SU_q\left( 2\right) \right) \to B\left(\ell^2\left(\N^0 \right) \otimes L^2\left(\R\right) \right)$ given by

\begin{equation}\label{s_2_q_repr_eqn}
\begin{split}
\al \mapsto S\sqrt{1 - Q^2} \otimes 1_{B\left(L^2\left(\R\right) \right) }, \\
\bt \mapsto Q \otimes R_\R.
\end{split}
\end{equation}

\subsection{Finite-fold coverings}

\paragraph*{}
If $R^{\frac{1}{n}}_\R \in B\left(  L^2\left(\R \right)\right) $ is given by
$$
R^{\frac{1}{n}}_\R\left(  \xi\right)  = e^{\frac{2\pi ix}{n}}\xi.
$$ 
then $\left(R^{\frac{1}{n}}_\R \right)^n =  R_\R$. If $\widetilde{q} = \sqrt[n]{q}$ and 
$$
\widetilde{\bt} = \sum_{k = 0}^{\infty} \widetilde{q}^k p^\al_k\otimes R^{\frac{1}{n}}_\R\in B\left(\ell^2\left(\N^0 \right) \otimes L^2\left(\R\right) \right)  
$$
then $\widetilde{\bt}^n = \bt$. Denote by $C\left(\SU_q\left( 2\right) \right) \left[\widetilde{\bt} \right]$ a $C^*$-subalgebra of  $B\left(\ell^2\left(\N^0 \right) \otimes L^2\left(\R\right) \right)$ generated by $C\left(\SU_q\left( 2\right) \right) \bigsqcup \left\{\widetilde{\bt}\right\}$. Denote by $M\left[\widetilde{\bt} \right] \subset C\left(\SU_q\left( 2\right) \right) \left[\widetilde{\bt} \right]$ a free module left  $C\left(\SU_q\left( 2\right) \right)$ module given by
$$
M\left[\widetilde{\bt} \right] = \bigoplus_{j = 0}^{n-1} C\left(\SU_q\left( 2\right) \right) \widetilde{\bt}^j.
$$
If $j \in \{0,..., n-1\}$, $j \in \N^0$ then from $p^\al_k \in  C\left(\SU_q\left( 2\right) \right)$ it follows that $p^\al_k \widetilde{\bt}^j= p^\al_k \widetilde{q}^{-k}\left(R^{\frac{1}{n}}_\R \right)^j \in M\left[\widetilde{\bt} \right]$, hence $p^\al_k\left(R^{\frac{1}{n}}_\R \right)^j \in M\left[\widetilde{\bt} \right]$. Moreover if $\left\{z_k \in \C\right\}_{k \in \N^0}$ then from $\lim_{k \to \infty}z_k = 0$ it turns out
$$
\sum_{k = 0}^{\infty} z_kp^\al_k\left(R^{\frac{1}{n}}_\R \right)^k \in M\left[\widetilde{\bt} \right].
$$

Following conditions hold:

\begin{equation*}
\begin{split}
\widetilde{\bt}^j \al = \left( \sum_{k=0}^{\infty}\widetilde{q}^{jk} p^\al_k\otimes \left(R^{\frac{1}{n}}_\R \right)^j\right)\left(  S \sqrt{1 - Q^2} \otimes 1 \right) = \left(  S \sqrt{1 - Q^2} \otimes 1 \right) \left( \sum_{k=0}^{\infty}\widetilde{q}^{j\left( k+1\right) } p^\al_k\otimes \left(R^{\frac{1}{n}}_\R \right)^j\right)
\end{split}
\end{equation*}
From $\lim_{k \to \infty}\widetilde{q}^{j\left( k+1\right) } = 0$ it turns out $\widetilde{\bt}^j \al$ lies in $M\left[\widetilde{\bt} \right]$. Similarly we have $\widetilde{\bt}^j \al^*\in \left[\widetilde{\bt} \right]$ it follows that
$$
M\left[\widetilde{\bt} \right] =C\left(\SU_q\left( 2\right) \right) \left[\widetilde{\bt} \right],
$$
i.e. $C\left(\SU_q\left( 2\right) \right) \left[\widetilde{\bt} \right]$ is a  finitely generated free  $C\left(\SU_q\left( 2\right) \right)$-module 
\begin{equation}\label{su_q_2_dir_sum_eqn}
C\left( \SU_q\left(2\right) \right) \left[\widetilde{\bt} \right] = \bigoplus_{j = 0}^{n-1} C\left(\SU_q\left( 2\right) \right) \widetilde{\bt}^j= C\left(\SU_q\left( 2\right) \right)^n
\end{equation}
There is the action of $\Z_n$ on $C\left(\SU_q\left( 2\right) \right) \left[\widetilde{\bt} \right]$ given by
\begin{equation*}
\begin{split}
\overline{m} a \widetilde{\bt}^k = e^{\frac{2 \pi i mk}{n}} a \widetilde{\bt}^k; \text{ where } a \in C\left(\SU_q\left( 2\right) \right), ~\overline{m} \in \Z_n, ~m \in \Z \text{ is representative of } \overline{m}.
\end{split}
\end{equation*}

The above construction gives a following result.
\begin{thm}\cite{ivankov:qnc}
	The triple $\left( C\left( \SU_q\left( 2\right)\right), C\left( \SU_q\left( 2\right)\right)\left[\widetilde{\bt}\right], \Z_n\right)$ is an unital noncommutative finite-fold covering.
\end{thm}

\subsection{The structure of the covering algebra}\label{cov_alg_str_sec}
\paragraph*{}
From the above construction it follows that
$$
\widetilde{\widetilde{\bt}}= \sum_{j = 0}^{\infty} q^j p^\al_j \otimes R^{\frac{1}{n}}_\R \in C\left( \SU_q\left( 2\right)\right)\left[\widetilde{\bt}\right].
$$
Direct calculations shows that
\begin{equation*}
\begin{split}
\al^*\al + \widetilde{\widetilde{\bt}}^*\widetilde{\widetilde{\bt}} = 1, ~~ \al\al^* + q^2\widetilde{\widetilde{\bt}}\widetilde{\widetilde{\bt}}^* =1,
\\
\al\widetilde{\widetilde{\bt}} - q \widetilde{\widetilde{\bt}}\al = 0, ~~\al\widetilde{\widetilde{\bt}}^*-q\widetilde{\widetilde{\bt}}^*\al = 0,
\\
\widetilde{\widetilde{\bt}}^*\widetilde{\widetilde{\bt}} = \bt\widetilde{\widetilde{\bt}}^*.
\end{split}
\end{equation*}
Above relations coincide with \eqref{s_2_q_repr_eqn} it follows that there is a $*$-isomorphism given by
\begin{equation*}
\begin{split}
C\left( \SU_q\left( 2\right)\right)\xrightarrow{\approx}C\left( \SU_q\left( 2\right)\right)\left[\widetilde{\bt}\right],\\
\al \mapsto \al,~~
\bt\mapsto \widetilde{\widetilde{\bt}},
\end{split}
\end{equation*}
i.e. the covering algebra $C\left( \SU_q\left( 2\right)\right)\left[\widetilde{\bt}\right]$ is *-isomorphic to the base algebra $C\left( \SU_q\left( 2\right)\right)$.

\subsection{Symmetry and grading}
\paragraph*{}
Let $\A\subset SU_q\left( 2\right)$ is a dense subalgebra which is generated by $\al, \al^*, \bt, \bt^*$ as an abstract algebra.
\begin{thm}\cite{woronowicz:su2}
The set of all elements of the form
\begin{equation}\label{su_q_2_basis_eqn}
\al^k\bt^n\bt^{*m} \text{ and } \al^{*k'}\bt^n\bt^{*m}
\end{equation}
where $k, m, n = 0, 1, 2, \dots,~k'=1,2, \dots$ forms a basis in $\A$: any element of $\A$ can be written in the unique way as a finite linear combination of elements \eqref{su_q_2_basis_eqn}.
\end{thm}
From the above theorem there is an action  of $U\left( 1\right)$ on $\A$ given by
$$
g \left( \al^k\bt^n\bt^{*m} \right)  = \varphi(g)_{\C^\times}^{n - m}\al^{*k'}\bt^n\bt^{*m} \text{ and } g \left( \al^{*k'}\bt^n\bt^{*m}\right)   = \varphi_{\C^\times}\left(g \right) ^{n - m}\al^{*k'} \al^{*k'}\bt^n\bt^{*m}
$$
where $g \in U\left(1 \right)$ and $ \varphi_{\C^\times}: U\left( 1\right) \to \C^\times$ the natural homomorphism from $U\left(1 \right)$ to the multiplicative group of complex numbers. There is a $\Z$-grading 
$$
\A = \bigoplus_{j \in \Z} \A_j
$$
such that $a \in \A_j$ is equivalent to
$$
ga = \varphi(g)^j_{\C^\times}a \text { for any } g \in U\left( 1\right).
$$  
It turns out
$$
\al^k\bt^n\bt^{*m} \text{ and } \al^{*k'}\bt^n\bt^{*m} \text{ lie in }  \A_{n - m}.
$$
Let $\left( C\left( \SU_q\left( 2\right)\right), C\left( \SU_q\left( 2\right)\right)\left[\widetilde{\bt}\right], \Z_n\right)$ be a covering projection. From $\widetilde{\beta}^n = \beta$ and \eqref{su_q_2_dir_sum_eqn} it follows that there is the natural $\Z$-grading on $C\left( \SU_q\left( 2\right)\right)\left[\widetilde{\bt}\right]$
given by
$$
a \beta^j \in \left( C\left( \SU_q\left( 2\right)\right)\left[\widetilde{\bt}\right]\right)_{nk + j} \text{ where } a \in C\left( \SU_q\left( 2\right)\right)_k
$$ 
where subscripts mean the grading.

\subsection{Contradiction}
\paragraph*{}
Suppose that there is a structure of quantum group $\left( C\left(\SU_q\left( 2\right)\right)\left[\widetilde{\bt}\right], \widetilde{\Delta} \right)$ which satisfies to the Lemma \ref{comm_lem}.  From $\bt = \widetilde{\bt}^n$,  \eqref{su_2_qgr_eqn}, and the condition (i) of the Lemma \ref{comm_lem} it turns out 
\begin{equation}\label{bt_n_eqn}
\left(\widetilde{\Delta}\left(\widetilde{\bt} \right)  \right)^n = \Delta\left( \bt\right)  = \beta\otimes\alpha+\alpha^*\otimes\beta = \widetilde{\bt}^n\otimes\alpha+\alpha^*\otimes\widetilde{\bt}^n. 
\end{equation}
Denote by $$
D \stackrel{\text{def}}{=} C\left( \SU_q\left( 2\right)\right)\left[\widetilde{\bt}\right] \otimes C\left( \SU_q\left( 2\right)\right)\left[\widetilde{\bt}\right].
$$
The $\Z$-grading on $C\left( \SU_q\left( 2\right)\right)\left[\widetilde{\bt}\right]$ induces  the natural $\Z\times\Z$ grading on  $D = C\left( \SU_q\left( 2\right)\right)\left[\widetilde{\bt}\right] \otimes C\left( \SU_q\left( 2\right)\right)\left[\widetilde{\bt}\right]$.
Clearly 
\begin{equation*}
\begin{split}
\widetilde{\bt}^n\otimes\alpha \in D_{(n,0)},\\
\alpha^*\otimes\widetilde{\bt}^n \in D_{(0,n)}
\end{split}
\end{equation*}
where subscripts $(n, 0)$ and $(0, n)$ mean grading.
Suppose that
$$
\widetilde{\Delta}\left(\widetilde{\bt} \right) = \sum_{\left( j, k\right)  \in \Z\times \Z} a_{jk}
$$
where
$$
a_{jk} \in D_{(j,k)}.
$$
Let $j_{\text{max}} \in \Z$ be a maximal number such that there is $k \in \Z$ which satisfy to the condition $a_{j_{\text{max}}, k} \neq 0$. The inequality $j_{\text{max}} > 1$ contradicts with \eqref{bt_n_eqn} because right part of \eqref{bt_n_eqn} does not contain summands in
	$
	D_{\left( nj_{\text{max}},k\right) }
	$.
	Similarly one can prove that the minimal value $j_{\text{min}}$ of $j$ such that $a_{j_{\text{min}}, k} \neq 0$ satisfies to an inequality $j_{\text{min}} \ge 0$. Using the same arguments one can prove that if $a_{jk} \neq 0$ then $0 \ge k \ge 1$. In result one has
		$$
	\widetilde{\Delta}\left(\widetilde{\bt} \right) = a_{00} + a_{01} + a_{10} + a_{11}.
		$$
		If $a_{00} \neq 0$ then $\left(	\widetilde{\Delta}\left(\widetilde{\bt} \right) \right)^n \bigcap D^{(0,0)} \neq 0$ and from this contradiction it turns out
		$a_{00}=0$. Similarly $a_{11} = 0$.
		Following condition holds
\begin{equation*}
\begin{split}
	\widetilde{\Delta}\left(\widetilde{\bt} \right)^n=\left(a_{01}+a_{10} \right)^n = a_{01}^n + a_{10}^n + r,\\
 a_{01}^n \in D^{0,n},\\
 a_{10}^n \in D^{n, 0},\\
 r \notin   D^{0,n} \bigoplus  D^{n,0},
\end{split}
\end{equation*}
hence $a_{01}^n= \alpha^*\otimes\widetilde{\bt}^n$, $a_{01}^n= \alpha^*\otimes\widetilde{\bt}^n$. Otherwise
\begin{equation*}
\begin{split}
r = n a_{10} a_{01}^{n-1} + r'
\end{split}
\end{equation*}
where $n a_{10} a_{01}^{n-1} \in D^{(1,n-1)}$, $r' \notin D^{(1,n-1)}$. From $a_{10}^n a_{01}^{n}\neq 0$ it turns out $a_{10} a_{01}^{n-1}\neq 0$ hence $r \neq 0$.
It follows that
$$
\widetilde{\Delta}\left(\widetilde{\bt} \right)^n=\left(a_{01}+a_{10} \right)^n  \neq \widetilde{\bt}^n\otimes\alpha+\alpha^*\otimes\widetilde{\bt}^n.
$$
This contradiction proves that the quantum group $\left( C\left( \SU_q\left( 2\right)\right), \Delta\right) $ and the finite-fold noncommutative covering $\left( C\left( \SU_q\left( 2\right)\right), C\left( \SU_q\left( 2\right)\right)\left[\widetilde{\bt}\right], \Z_n\right)$ projection do not satisfy to the Lemma \ref{comm_lem}.
\begin{rem}
	From \ref{cov_alg_str_sec} it follows the *-isomorphism $C\left( \SU_q\left( 2\right)\right) \xrightarrow{\approx}C\left( \SU_q\left( 2\right)\right)\left[\widetilde{\bt}\right]$, hence there is a structure of quantum group on $C\left( \SU_q\left( 2\right)\right)\left[\widetilde{\bt}\right]$. However in contrary to the commutative case this structure does not naturally follow from the structure of the quantum group $\left( C\left( \SU_q\left( 2\right)\right), \Delta\right)$ and the noncommutative finite-fold covering projection $$\left( C\left( \SU_q\left( 2\right)\right), C\left( \SU_q\left( 2\right)\right)\left[\widetilde{\bt}\right], \Z_n\right).$$
\end{rem}

\section{Conclusion}
\paragraph{} There is a set of geometrical statements which have noncommutaive generalizations, e.g. in \cite{ivankov:qncstr} it is proven a noncommutative analog of the theorem about a covering projection of a Riemannian manifold.
The  described in the Section  \ref{counter_sec} counterexample proves that the  analogy between coverings of topological groups and quantum groups is not full. However coverings of quantum groups satisfy to the Theorem \ref{main_thm} which is weaker than the Lemma \ref{comm_lem} about coverings of commutative quantum groups.

\end{document}